\documentclass[11pt]{amsart}

\usepackage[a4paper,hmargin=3.5cm,vmargin=4cm]{geometry}
\usepackage{amsfonts,amssymb,amscd,amstext}
\usepackage{graphicx}
\usepackage[dvips]{epsfig}

\usepackage{verbatim}
\usepackage{fancyhdr}
\input xy
\xyoption{all}

\usepackage{times}

\usepackage{enumerate}
\usepackage{titlesec}
\usepackage{mathrsfs}

\titleformat{\section}
{\filcenter\bfseries\large} {\thesection{.}}{0.2cm}{}
\titleformat{\subsection}[runin]
{\bfseries} {\thesubsection{.}}{0.15cm}{}[.]
\titleformat{\subsubsection}[runin]
{\em}{\thesubsubsection{.}}{0.15cm}{}[.]

\usepackage[up,bf]{caption}

\newtheorem{theorem}{Theorem}[section]

\newtheorem{corollary}[theorem]{Corollary}

\theoremstyle{definition}

\newtheorem{conjecture}[theorem]{Conjecture}

\numberwithin{equation}{section}
\numberwithin{figure}{section}

\newcommand\Cscr{\mathscr{C}}

\newcommand\Oscr{\mathscr{O}}

\newcommand\B{\mathbb{B}}
\newcommand\C{\mathbb{C}}

\newcommand\N{\mathbb{N}}

\newcommand\R{\mathbb{R}}

\renewcommand\b{\mathbb{B}}
\renewcommand\c{\mathbb{C}}

\renewcommand\d{\mathbb D}

\newcommand\n{\mathbb{N}}
\renewcommand\r{\mathbb{R}}

\newcommand\igot{\mathfrak{i}}

\renewcommand\igot{\mathfrak{i}}

\newcommand\ggot{\mathfrak{g}}

\renewcommand\imath{\igot}

\newcommand\hra{\hookrightarrow}

\newcommand\wt{\widetilde}

\newcommand\dist{\mathrm{dist}}

\newcommand\Aut{\mathrm{Aut}}

\def\dist{\mathrm{dist}}

\usepackage{color}

\begin{document}

\fancyhead[LO]{Complete densely embedded complex lines in $\C^2$}
\fancyhead[RE]{A.\ Alarc\'on and F.\ Forstneri\v c} 
\fancyhead[RO,LE]{\thepage}

\thispagestyle{empty}

\vspace*{1cm}
\begin{center}
{\bf\LARGE Complete densely embedded complex lines in $\C^2$}

\vspace*{0.5cm}

{\large\bf  Antonio Alarc{\'o}n \; and \; Franc Forstneri{\v c}} 
\end{center}


\vspace*{1cm}

\begin{quote}
{\small
\noindent {\bf Abstract}\hspace*{0.1cm}
In this paper we construct a complete injective holomorphic immersion $\C\to\C^2$ whose image is dense in $\C^2$. 
The analogous result is obtained for any closed complex submanifold $X\subset \C^n$
for $n>1$ in place of $\C\subset\C^2$. We also show that, if $X$ intersects the unit ball $\b^n$ of $\c^n$ 
and $K$ is a connected compact subset of $X\cap\b^n$, then there is a Runge domain $\Omega\subset X$ 
containing $K$ which admits a complete holomorphic embedding $\Omega\to\b^n$ 
whose image is dense in $\b^n$.

\vspace*{0.2cm}

\noindent{\bf Keywords}\hspace*{0.1cm} complete complex submanifold, holomorphic embedding 

\vspace*{0.1cm}

\noindent{\bf MSC (2010):}\hspace*{0.1cm} 32H02; 32E10, 32M17, 53A10}
%
%
\end{quote}


\vspace{1.5mm}

\section{Introduction} 
\label{sec:intro}

A smooth immersion $f\colon X\to\r^n$ from a smooth manifold $X$ to the Euclidean space $\r^n$ is said to be {\em complete} if the image of every divergent path in $X$ has infinite length in $\R^n$; equivalently, if the metric $f^*(ds^2)$ on $X$ induced by the Euclidean metric $ds^2$ on $\R^n$ is complete. 
An injective immersion will be called an {\em embedding}.
If $X$ is an open Riemann surface, $n\ge 3$, and $f\colon X\to\r^n$ is a conformal immersion,
then it parametrizes a minimal surface in $\R^n$ if and only if it is a harmonic map. 

A seminal result of Colding and Minicozzi \cite[Corollary 0.13]{ColdingMinicozzi2008AM} 
states that a complete embedded minimal surface of finite topology in $\r^3$ is necessarily proper in $\R^3$;
this was extended to surfaces of finite genus and countably many ends by Meeks, P\'erez, and Ros \cite{MeeksPerezRos-CY}.
This is no longer true for complex curves in $\c^2$ (a special case of minimal surfaces in $\r^4$). 
Indeed, there exist complete embedded complex curves in $\c^2$ with arbitrary topology which are bounded and hence 
non-proper (see \cite{AlarconGlobevnik2016C2}; the case of finite topology was previously shown in 
\cite{AlarconGlobevnikLopez2016Crelle}). Furthermore, every relatively compact domain in $\c$ admits a complete 
non-proper holomorphic embedding into $\c^2$ (see \cite[Corollary 4.7]{AlarconFernandezLopez2013CVPDE}). 
Since all examples in the cited sources are normalized by open Riemann surfaces of {\em hyperbolic} type 
(i.e., carrying non-constant negative subharmonic functions; 
see e.\ g.\ \cite[p.\ 179]{FarkasKra1992Springer}), one is led to wonder whether hyperbolicity plays a role in this context. 
The purpose of this note is to show that it actually does not.
The following is our first main result.

\begin{theorem}\label{th:main}
Given a closed complex submanifold $X$ of $\C^n$ for some $n>1$, there
exists a complete holomorphic embedding $f\colon X\to\C^n$ 
such that $f(X)$ contains any given countable subset of $\c^n$.
In particular, $f(X)$ can be made dense in $\C^n$.
\end{theorem}

By {\em dense} we shall always mean {\em everywhere dense}.
Note that if $f(X)$ is dense in $\C^n$ then $f\colon X\to\c^n$ 
is non-proper. Taking $n=2$ and $X=\C$ gives the following corollary.

\begin{corollary}\label{cor:main}
There is a complete embedded complex line $\C\to\C^2$ with a dense image.
\end{corollary}

Corollary \ref{cor:main} also holds if $\C$ is replaced by any open Riemann surface admitting
a proper holomorphic embedding into $\C^2$. There are many open {\em parabolic} (i.e., non-hyperbolic) 
Riemann surfaces enjoying this condition; it is however not known whether all open Riemann surfaces do.
For a survey of this classical embedding problem we refer to Sections 8.9 and 8.10 in \cite{Forstneric2011} 
and the paper \cite{ForstnericWold2013APDE}. Without taking care of injectivity, every open Riemann surface 
admits complete dense holomorphic immersions into $\c^n$ for any $n\ge 2$ and complete dense 
conformal minimal immersions into $\r^n$ for $n\ge 3$ (see \cite{AlarconCastro2016Pre1}).

These results provide additional evidence that there is much more
room for conformal minimal surfaces (even those given by holomorphic maps)
in $\R^4=\C^2$ than in $\R^3$. We point out that it is quite easy 
to find injective holomorphic immersions $\C\to\C^2$ which are neither 
complete nor proper. For example, if $a>0$ is irrational then the map $\C\ni z\mapsto (e^z,e^{az})\in\C^2$ 
is an injective immersion, but the image of the negative real axis is a curve of finite length in $\C^2$
terminating at the origin. On the other hand, it is an open problem whether a
conformal minimal embedding $\C\to \R^3$ is necessarily proper; see
\cite[Conjecture 1.2]{FonteneleXavier2010}.

To prove Theorem \ref{th:main}, we use an idea from the recent paper by Alarc\'on, Globevnik, and L\'opez 
\cite{AlarconGlobevnikLopez2016Crelle}. The construction relies on two ingredients. First, in any spherical 
shell in $\C^n$ one can find a  compact polynomially convex set $L$, consisting of finitely many
pairwise disjoint balls contained in affine real hyperplanes,
such that any curve traversing this shell and avoiding $L$ has length bigger than a prescribed constant.
For a suitable choice of $L$ with this property it is then possible to find a holomorphic automorphism
of $\C^n$ which pushes a given complex submanifold $X\subset\C^n$ off $L$.
The construction of such an automorphism uses the main result of the {\em Anders\'en-Lempert theory}. 
In \cite{AlarconGlobevnikLopez2016Crelle} this construction was used to show that every closed complex 
submanifold $X\subset \C^n$ contains a bounded Runge domain
$\Omega$ admitting a proper complete holomorphic embedding
into the unit ball of $\C^n$; furthermore, $\Omega$ can be chosen to contain any given
compact subset of $X$. Clearly, such $\Omega$ carries nonconstant negative 
plurisubharmonic functions and is Kobayashi hyperbolic, so in general one cannot map all of $X$ into the ball. 
We choose instead a sequence of automorphisms which converges uniformly on compacts in $X$
to a complete holomorphic embedding $X\hra \C^n$ whose image contains a prescribed countable set
of points in $\C^n$.

It is natural to ask whether the analogue of Theorem \ref{th:main} holds for more general target manifolds in place of $\C^n$. Since our proof relies on the Anders\'en-Lempert theory  which holds on any Stein manifold $Y$
enjoying Varolin's {\em density property} (the latter meaning that every holomorphic vector field on $Y$ 
can be approximated uniformly on compacts by Lie combinations of $\C$-complete holomorphic vector fields;
see Varolin \cite{Varolin2001} or \cite[Sec.\ 4.10]{Forstneric2011}),
the following is a reasonable conjecture.

\begin{conjecture} 
Assume that $Y$ is a Stein manifold with the density property.
Choose a complete Riemannian metric $\ggot$ on $Y$. 
\begin{enumerate}[\rm (a)] 
\item If $\dim Y\ge 3$ then there exists a $\ggot$-complete holomorphic embedding $\C\to Y$ with a dense image. 
\vspace{1mm}
\item More generally, if $X$ is a Stein manifold, $\dim X< \dim Y$, and there is a proper holomorphic embedding $X\hra Y$, then there exists a $\ggot$-complete injective holomorphic immersion $X \to Y$ with a dense image.
\end{enumerate}
\end{conjecture}
It was recently shown in \cite{Andrist-etall2016} that, if $X$ and $Y$ are as in assertion (b) above and satisfy 
$2\dim X+1\le \dim Y$, then there exists a proper (hence complete) holomorphic embedding $X\hra Y$. Thus, 
for such dimensions we are just asking whether {\em proper} can be replaced by {\em dense}, keeping completeness.

It is known that for any $n>1$ the unit ball $\b^n$ of $\c^n$ contains complete properly embedded complex hypersurfaces 
(see \cite{AlarconLopez2016JEMS,Globevnik2015AM,AlarconGlobevnikLopez2016Crelle} and the references therein); 
this settles in an optimal way a problem posed by Yang in 1977 about the existence of complete bounded complex submanifolds 
of $\c^n$ (see \cite{Yang1977,Yang1977JDG}). Moreover, given a discrete subset $\Lambda\subset\b^2$ there are complete 
properly embedded complex curves in $\b^2$ containing $\Lambda$ (see \cite{Globevnik2016JMAA} for discs and 
\cite{AlarconGlobevnik2016C2} for examples with arbitrary topology).
It remained and open problem whether $\b^n$ also admits complete densely embedded complex submanifolds.
Our second main result gives an affirmative  answer to this question.

%
%
\begin{theorem}\label{th:ball}
Let $X$ be a closed complex submanifold of $\c^n$ for some $n>1$ such that $X\cap \b^n\neq\varnothing$. 
Given a connected compact subset $K\subset X\cap\b^n$, there are a pseudoconvex Runge domain $\Omega\subset X$ containing $K$ and a complete holomorphic embedding 
$f\colon \Omega\to\b^n$ whose image $f(\Omega)$ contains any given countable subset of $\b^n$. 
In particular, $f(\Omega)$ can be made dense in $\b^n$.
\end{theorem}

As above, if $f(\Omega)\subset\b^n$ is dense then the map $f\colon \Omega\to\b^n$ is non-proper. 
Taking $n=2$ and $X=\d:=\{\zeta\in\c\colon |\zeta|<1\}$ we obtain the following corollary.

\begin{corollary}\label{cor:ball}
There is a complete embedded complex disc $\d\to\b^2$ with a dense image.
\end{corollary}

More generally, it follows from Theorem \ref{th:ball} 
that in $\b^2$ there are complete embedded complex curves with arbitrary finite topology and containing 
any given countable subset. (See Corollary \ref{cor:ball-2}.)
Without taking care of injectivity, given an arbitrary domain (i.e., a connected open subset) $D$ in $\c^n$ $(n\ge 2)$,
on each open connected orientable smooth surface there is a complex structure such that the resulting open 
Riemann surface admits complete dense holomorphic immersions into $D$; moreover, every bordered Riemann 
surface carries a complete holomorphic immersion into $D$ with dense image (see \cite{AlarconCastro2016Pre1}). 
The analogous results for conformal minimal immersions into any domain in $\r^n$ ($n\ge 3$) also hold 
(see \cite{AlarconCastro2016Pre1}). 

The proof of Theorem \ref{th:ball} uses arguments similar to those in the proof of Theorem \ref{th:main}, 
but with an additional ingredient to keep the image of the embedding $f$ inside the ball.

\subsection*{Notation} Given a closed complex submanifold $X$ of $\C^n$ $(n>1)$, a compact set $K\subset X$, 
and a map $f=(f_1,\ldots,f_n) \colon X\to\C^n$,
we write $\|f\|_K=\sup\{|f(x)|:x\in K\}$ where $|f|^2=\sum_{j=1}^n |f_j|^2$.
Denote by $ds^2$ the Euclidean metric on $\C^n$.
Given an immersion $f\colon X\to\C^n$, we denote by $\dist_f(x,y)$ the distance 
between points $x,y\in X$ in the metric $f^*(ds^2)$ on $X$. 
If $K\subset L$ are compact subsets of $X$, we set
\begin{equation}\label{eq:distf}
	\dist_f(K,X\setminus L) = \inf\{\dist_f(x,y): x\in K,\ y\in X\setminus L\}.
\end{equation}


\section{Proof of Theorem \ref{th:main}} 
\label{sec:proof}  

Let  $X$ be a closed complex submanifold of $\C^n$ for some $n>1$ and let $A=\{a_j\}_{j\in\N}$ be any 
countable subset of $\C^n$. Pick a compact $\Oscr(X)$-convex set $K_0\subset X$ and a number $0<\epsilon_0<1$. 
Let $f_0$ denote the inclusion map $X\hra \C^n$. In order to prove Theorem \ref{th:main}, we shall 
inductively construct the following:
\begin{itemize}
\item[\rm (a)] an exhaustion of $X$ by an increasing sequence of compact $\Oscr(X)$-convex sets
\[
	K_1\subset K'_2\subset K_2 \subset K'_3\subset K_3 \subset \cdots \subset \bigcup_{i=1}^\infty K_i = X
\]
such that $K_{i-1}\subset \mathring K'_{i}$ and $K'_{i}\subset \mathring K_{i}$ hold for all $i\in\N$,
\vspace{1mm}
\item[\rm (b)]
a sequence of proper holomorphic embeddings $f_i\colon X\hra \C^n$ $(i\in\N)$, 
\vspace{1mm}
\item[\rm (c)]  a discrete sequence of points $\{b_i\}_{i\in\N}\subset X$ with $b_i\in K_i$ for every $i\in \N$, and
\vspace{1mm}
\item[\rm (d)] a decreasing sequence of numbers $\epsilon_i>0$,  
\end{itemize}
such that the following conditions hold for every $i\in\N$:
\begin{itemize}
\item[\rm (i)] $\|f_{i}-f_{i-1}\|_{K_{i-1}} <\epsilon_{i-1}$,
\vspace{1mm}
\item[\rm (ii)] $a_j=f_i(b_j)  \in f_i(K_i)$ for $j=1,\ldots,i$ and $f_i(b_j)=f_{i-1}(b_j)$ for $j=1,\ldots,i-1$,
\vspace{1mm}
\item[\rm (iii)] $\dist_{f_i}(K_{i-1},X\setminus K'_i) > 1/\epsilon_{i-1}$ (see \eqref{eq:distf}),
\vspace{1mm}
\item[(iv)] $0<\epsilon_i <\epsilon_{i-1}/2$,
\vspace{1mm}
\item[(v)]  if $g\colon X\to \C^n$ is a holomorphic map such that $\|g-f_i\|_{K_i}< 2\epsilon_i$, then
$g$ is an injective immersion on $K_{i-1}$ and  $\dist_{g}(K_{i-1},X\setminus K_i) > 1/(2\epsilon_{i-1})$.
\end{itemize}

Assume for a moment that sequences with these properties exist. Conditions (a) and (iv) ensure that 
the sequence $f_i$ converges uniformly on compacts in $X$ to a holomorphic map
$f=\lim_{i\to\infty} f_i\colon X\to \C^n$. By (i) and (iv) we have for every $i\in\N$ that
\[
	\|f-f_i\|_{K_i}\le \sum_{k=i}^\infty \|f_{k+1}-f_k\|_{K_i} < \sum_{k=i}^\infty \epsilon_k <2\epsilon_i.
\]
Hence condition (v) implies that $f$ is an injective immersion on $K_{i-1}$ and
\[
	\dist_{f}(K_{i-1},X\setminus K_i) > 1/(2\epsilon_{i-1}). 
\]
Since this holds for every $i\in \N$ and $\sum_i 1/\epsilon_i=+\infty$, 
it follows that $f\colon X\to\C^n$ is a complete injective immersion. Finally, condition (ii) implies
that $f(X)$ contains the set $A=\{a_j\}_{j\in\N}$.
This completes the proof.

Let us now explain the induction. 
We shall frequently use the well known fact that if $g\colon X\hra \C^n$ is a proper holomorphic embedding
and $K\subset X$  is a compact $\Oscr(X)$-convex set,  then the set $g(K)\subset\C^n$ is polynomially convex.

Assume that for some $i\in\N$ we have found maps $f_j$, sets $K'_j\subset K_j$ and numbers 
$\epsilon_j$ satisfying the stated conditions for $j=0,\ldots,i-1$.
The next map $f_i$ will be of the form $f_i=\Phi\circ f_{i-1}$ for some holomorphic
automorphism $\Phi\in\Aut(\C^n)$ which will be found in two steps,
\[
	\Phi=\phi\circ \theta \quad\text{with}\quad \phi,\theta\in\Aut(\C^n).
\]

Let $\B=\B^n$ be the open unit ball in $\C^n$. Choose a number $r>0$ such that 
\[
	f_{i-1}(K_{i-1})\subset r\B,
\]
and then pick numbers $R>r'$ with $r'>r$.
In the open spherical shell $S=R\,\B\setminus r'\overline\B$ we choose a 
labyrinth $L=\bigcup_{k=1}^\infty L_k$ of the type constructed in 
\cite[Theorem 2.5]{AlarconGlobevnikLopez2016Crelle}, i.e., every set $L_k$ is a ball in an
affine real hyperplane $\Lambda_k\subset \C^n$ such that these balls are pairwise
disjoint, the set $\wt L_k = \bigcup_{j=1}^{k}L_j$ is contained in an open half-space determined
by $\Lambda_{k+1}$ for every $k\in\N$, and any path $\lambda\colon [0,1)\to R\, \B\setminus L$ with 
$\lambda(0)\in r'\overline\B$ and $\lim_{t\to 1}|\lambda(t)|=R$ has infinite Euclidean length. 
(Alternatively, we may use a labyrinth of the type constructed by Globevnik in \cite[Corollary 2.2]{Globevnik2015AM}.)
It follows that $\wt L_k\cap r'\overline \B=\varnothing$ and $\wt L_k \cup r'\overline \B$ 
is polynomially convex for every $k\in\N$. 
Fix $k_0\in\N$ big enough such that every path $\lambda\colon [0,1]\to \C^n\setminus \wt L_{k_0}$
with $\lambda(0)\in r' \overline\B$ and $\lambda(1)\in \C^n\setminus R\,\B$ has length
bigger than $1/\epsilon_{i-1}$. Choose a holomorphic automorphism $\theta\in\Aut(\c^n)$ satisfying
the following conditions:
\begin{itemize}
\item[\rm (I)] $|\theta(f_{i-1}(x)) - f_{i-1}(x)| < \min\{\epsilon_{i-1}/2,r'-r\}$ for all $x\in K_{i-1}$,
\vspace{1mm}
\item[\rm (II)] $\theta(a_j) = a_j$ for $j=1,\ldots,i-1$
(note that $a_j=f_{i-1}(b_j)\in f_{i-1}(K_{i-1})$ for $j=1,\ldots,i-1$),
\vspace{1mm}
\item[\rm (III)] $a_i\notin \theta(f_{i-1}(X))$, and 
\vspace{1mm}
\item[\rm (IV)] $\theta(f_{i-1}(X)) \cap \wt L_{k_0}=\varnothing$.
\end{itemize}
Such $\theta$ is found by an application of the Anders\'en-Lempert theory as explained
in \cite[Proofs of Lemma 3.1 and Theorem 1.6]{AlarconGlobevnikLopez2016Crelle}, using the fact that 
the set $f_{i-1}(K_{i-1})\cup \wt L_{k_0}$ is polynomially convex (since $f_{i-1}(K_{i-1})\subset r\overline\B$
and  $r\overline\B \cup \wt L_{k_0}$ is polynomially convex).
The explicit result used in their proof is  \cite[Theorem 2.1]{ForstnericRosay1993}
which is also available in \cite[Theorem 4.12.1]{Forstneric2011}. 

Consider the proper holomorphic embedding $g_i=\theta\circ f_{i-1}\colon X\hra\C^n$. The compact set
\[
	K'_{i}=\{x\in X : |g_i(x)| \le R+1\}
\]
is $\Oscr(X)$-convex and contains $K_{i-1}$ in its interior. By condition (I) we have 
$g_i(K_{i-1})\subset r'\B$, and hence condition (IV) and the choice of $k_0$ imply
\[
	\dist_{g_i}(K_{i-1},X\setminus K'_i) > 1/\epsilon_{i-1}.
\]

Choose a point $b_i\in X\setminus K'_i$. The set $K'_i\cup\{b_i\}$ is then $\Oscr(X)$-convex, 
and hence its image $g_i(K'_i) \cup \{g_i(b_i)\}\subset g_i(X)\subset\C^n$ is polynomially convex.
By the Anders\'en-Lempert theorem (see  \cite[Theorem 2.1]{ForstnericRosay1993}  or \cite[Theorem 4.12.1]{Forstneric2011}) 
we can find an automorphism $\phi\in\Aut(\C^n)$ which approximates 
the identity map as closely as desired on $g_i(K'_i)$, it fixes each of the points
$a_1,\ldots, a_{i-1}\in g_i(K_{i-1})$, and it satisfies $\phi(g_i(b_i))=a_i$.
If the approximation is close enough, then the proper holomorphic embedding  
\[
	f_i=\phi\circ g_i = \phi\circ\theta \circ f_{i-1}\colon X\hra\C^n
\]
satisfies conditions (i), (ii) and (iii) for the index $i$. Indeed, (i) and (ii) are obvious,
and (iii) follows by observing that 
\[
	f_i(K_{i-1})\subset r'\B, \quad f_i(bK'_{i})\subset \C^n\setminus R\,\overline\B,
	\quad \text{and}\ \ f_i(K'_{i}) \cap \wt L_i=\varnothing
\]
provided that $\phi$ approximates the identity sufficiently 
closely on $g_i(K'_{i})$. Thus, any path in $X$ starting in $K_{i-1}$ and ending in 
$X\setminus K'_i$ is mapped by $f_i$ to a path in $\C^n\setminus \wt L_{k_0}$ starting in $r'\B$ and ending in 
$\C^n\setminus R\,\overline \B$, hence its length is bigger than $1/\epsilon_{i-1}$ by the 
choice of $\wt L_{k_0}$. 

We now choose a compact $\Oscr(X)$-convex set $K_i\subset X$ containing $K'_i\cup\{b_i\}$
in its interior. Furthermore, $K_i$ can be chosen as big as desired, thereby ensuring that
the sequence of these sets will exhaust $X$. By choosing $\epsilon_i>0$  small enough we obtain 
conditions (iv) and (v). Indeed, since the sets $K_{i-1}\subset K'_{i}$ are contained
in the interior of $K_i$, uniform approximation on $K_i$ gives approximation in 
the $\Cscr^1$-norm on $K'_{i}$ by the Cauchy estimates. 

This finishes the induction step and hence completes the proof of Theorem \ref{th:main}.


\section{Proof of Theorem \ref{th:ball} and Corollary \ref{cor:ball}} 
\label{sec:proof-ball} 

We begin with the proof of Theorem \ref{th:ball}.

Let  $X$ be a closed complex submanifold of $\C^n$ for some $n>1$, and
let $f_0\colon X\hra \C^n$ denote the inclusion map. 
Let $K\subset X\cap\b^n$ be a connected compact subset,  and let $A=\{a_j\}_{j\in\N}$ be a countable subset of $\b^n$. 
Pick a compact connected $\Oscr(X)$-convex set $K_0\subset X\cap\b^n$ containing $K$ and  a number $0<\epsilon_0<1$. 
Similarly to what has been done in the proof of Theorem \ref{th:main}, we shall inductively construct the following:
\begin{itemize}
\item[\rm (a)] an increasing sequence of connected compact $\Oscr(X)$-convex subsets of $X$,
\[
	K_1\subset K'_2\subset K_2 \subset K'_3\subset K_3 \subset \cdots 
\]
such that $K_{i-1}\subset \mathring K'_{i}$ and $K'_{i}\subset \mathring K_{i}\subset X$ hold for all $i\in\N$,
\vspace{1mm}
\item[\rm (b)]
a sequence of proper holomorphic embeddings $f_i\colon X\hra \C^n$ $(i\in\N)$, 
\vspace{1mm}
\item[\rm (c)]  a sequence $(b_i)_{i\in\N}\subset X$ without repetition such that $b_i\in K_i$ for every $i\in \N$, and
\vspace{1mm}
\item[\rm (d)] a decreasing sequence of numbers $\epsilon_i>0$,  
\end{itemize}
such that the following conditions hold for every $i\in\N$:
\begin{itemize}
\item[\rm (i)] $\|f_{i}-f_{i-1}\|_{K_{i-1}} <\epsilon_{i-1}$,
\vspace{1mm}
\item[\rm (ii)] $a_j=f_i(b_j)  \in f_i(K_i)$ for $j=1,\ldots,i$ and $f_i(b_j)=f_{i-1}(b_j)$ for $j=1,\ldots,i-1$,
\vspace{1mm}
\item[\rm (iii)] $\dist_{f_i}(K_{i-1},X\setminus K'_i) > 1/\epsilon_{i-1}$  (see \eqref{eq:distf}),
\vspace{1mm}
\item[(iv)] $0<\epsilon_i <\epsilon_{i-1}/2$,
\vspace{1mm}
\item[(v)]  if $g\colon X\to \C^n$ is a holomorphic map such that $\|g-f_i\|_{K_i}< 2\epsilon_i$, then
$g$ is an injective immersion on $K_{i-1}$ and  $\dist_{g}(K_{i-1},X\setminus K_i) > 1/(2\epsilon_{i-1})$, and
\vspace{1mm}
\item[(vi)] $f_i(K_i)\subset \b^n$.
\end{itemize}

The main novelty with respect to the the proof of Theorem \ref{th:main} is condition {\rm (vi)} which 
implies that the connected domain 
\begin{equation}\label{eq:Omega}
	\Omega=\bigcup_{i=1}^\infty K_i  \subset X
\end{equation}
may be a proper subset of $X$. Note that $\Omega$ is pseudoconvex and Runge in $X$ since each set
$K_i$ is $\Oscr(X)$-convex. Granted these conditions, we see as in the proof of Theorem \ref{th:main} 
that the  limit map $f:=\lim_{i\to\infty} f_i\colon \Omega\to\c^n$ exists and is  
a complete holomorphic embedding whose image $f(\Omega)$ contains the countable set $A$;
moreover, we have $f(\Omega)\subset\b^n$ in view of {\rm (vi)}. Thus, to complete the proof of 
Theorem \ref{th:ball} it remains to establish the induction.

For the inductive step we assume that for some $i\in\N$ we have already found maps $f_j$, sets $K'_j\subset K_j$, 
and numbers $\epsilon_j>0$ satisfying the stated conditions for $j=0,\ldots,i-1$. 
(This is vacuous for $i=1$.) The next map $f_i$ will be obtained in two steps, each obtained by 
a composition with a suitably chosen holomorphic automorphism of $\C^n$.

Write $\b=\b^n$. By compactness of the set $K_{i-1}$ and property {\rm (vi)} for the index $i-1$ there 
is a number $0<r<1$ such that 
\begin{equation}\label{eq:fi-1Ki-1}
	f_{i-1}(K_{i-1})\subset r\b.
\end{equation}
Pick a number $R\in(r,1)$. Let $L=\bigcup_{k=1}^\infty L_k \subset R\,\B\setminus r\overline \B$ be a labyrinth 
as in the proof of Theorem \ref{th:main}. Set $\wt L_k=\bigcup_{j=1}^k L_k$ for all $k\in\n$.
Pick $k_0\in\n$ such that the length of any path $\lambda\colon [0,1]\to \c^n\setminus \wt L_{k_0}$ with $|\lambda(0)|=r$ 
and $|\lambda(1)|=R$ is bigger than $1/\epsilon_{i-1}$. Reasoning as in the proof of Theorem \ref{th:main}, we find 
a holomorphic automorphism $\theta\in\Aut(\c^n)$ satisfying
\begin{itemize}
\item[\rm (I)] $|\theta(f_{i-1}(x)) - f_{i-1}(x)| < \epsilon_{i-1}/2$ for all $x\in K_{i-1}$,
\vspace{1mm}
\item[\rm (II)] $\theta(a_j) = a_j$ for $j=1,\ldots,i-1$,
\vspace{1mm}
\item[\rm (III)] $a_i\notin \theta(f_{i-1}(X))$, and 
\vspace{1mm}
\item[\rm (IV)] $\theta(f_{i-1}(X)) \cap \wt L_i =\varnothing$.
\end{itemize}
Moreover, by \eqref{eq:fi-1Ki-1} we may choose $\theta$ close enough to the identity on $f_{i-1}(K_{i-1})$ so that
\begin{itemize}
\item[\rm (V)] $g_i(K_{i-1})\subset r\b$, where $g_i:=\theta\circ f_{i-1}\colon X\hra\c^n$.
\end{itemize}
Since $g_i$ is a proper holomorphic embedding, there is a connected compact $\Oscr(X)$-convex set $K_i'\subset X$ 
such that $K_{i-1}\subset \mathring K_i'$ and
\begin{equation}\label{eq:gibKi}
		g_i(bK_i')\subset\b\setminus R\, \overline\b.
\end{equation}
(For example, fixing a number $\rho\in (R,1)$, we may choose $K'_i$ such that $g_i(K_i')$ 
is the connected component of the set $g_i(X)\cap \rho\, \overline\B$ which contains $g_i(K_{i-1})$.)
Properties \eqref{eq:fi-1Ki-1}, {\rm (IV)}, {\rm (V)}, and \eqref{eq:gibKi} ensure that
\begin{equation}\label{eq:distgi}
	\dist_{g_i}(K_{i-1},X\setminus K_i')>1/\epsilon_{i-1}.
\end{equation}
Let $U$ be the connected component of $\b\cap g_i(X)$ containing the set $g_i(K_i')$.
Set $V:=g_i^{-1}(U)\subset X$ and note that $K_i'\subset V$. 
Pick a point $b_i\in V\setminus K_i'$;  then $g_i(b_i) \in U \setminus g_i(K_i')$.
Choose a smooth embedded arc $\gamma\subset V\setminus\mathring K_i'$ having an endpoint $p$ in $K_i'$ 
and being otherwise disjoint from $K_i'$. Then, $g_i(\gamma)$ is an embedded arc in $\b$ 
having $g_i(p)\in g_i(K_i')$ as an endpoint and being otherwise disjoint from $g_i(K_i')$. 
Since the set $\b\setminus g_i(K_i')$ is path connected and  contains the point $a_i$ in view of {\rm (III)}, 
there exists a homeomorphism
\[
	F\colon g_i(K_i'\cup \gamma)\to g_i(K_i')\cup F(g_i(\gamma))\subset\c^n
\] 
which equals the identity on a neighborhood of $g_i(K_i')$ such that the arc $F(g_i(\gamma))$ is 
contained in $\b$, has $g_i(p)$ and $a_i$ as endpoints, and is otherwise disjoint from $g_i(K_i')$.
Since $K_i'$ is $\Oscr(X)$-convex, the set $g_i(K_i')\subset\c^n$ is polynomially convex. In this situation, 
\cite[Proposition, p.\ 560]{ForstnericGlobevnikStensones1996MA} (on {\em combing hair
by holomorphic automorphisms}; see also \cite[Corollary 4.13.5, p.\ 148]{Forstneric2011})
enables us to approximate $F$ uniformly on $g_i(K_i'\cup \gamma)$ by a holomorphic automorphism 
$\phi\in\Aut(\c^n)$ such that 
\begin{equation}\label{eq:phiai}
	\text{$\phi(a_j)=a_j$\ \ for $j=1,\ldots,i-1$ \; and \; $\phi(g_i(b_i))=a_i$.}
\end{equation}
Consider the proper holomorphic embedding
\[
	f_i:=\phi\circ g_i=\phi\circ\theta\circ f_{i-1}\colon X\hra\c^n.
\]
If the approximation of $F$ by $\phi$ is close enough uniformly on $g_i(K_i'\cup \gamma)$ then
the inclusion \eqref{eq:gibKi} and the maximum principle guarantee that 
\[
	f_i(K_i'\cup \gamma)=\phi(g_i(K_i'\cup \gamma))\subset \b.
\] 
Hence there is a connected compact 
$\Oscr(X)$-convex subset $K_i\subset X$ such that $K_i'\cup \gamma\subset \mathring K_i$ and $f_i(K_i)\subset\b$. 
Assuming that the approximation of $F$ by $\phi$ is close enough, the inequality \eqref{eq:distgi} ensures that 
$\dist_{f_i}(K_{i-1},X\setminus K_i')>1/\epsilon_{i-1}$, and so the same holds when replacing $K_i'$ by the 
bigger set $K_i$. Summarizing, the map $f_i$ satisfies conditions {\rm (i)}, {\rm (iii)}, and {\rm (vi)}. Moreover,
conditions {\rm (II)}, \eqref{eq:phiai}, and the fact that $b_i\in \gamma\subset K_i$ guarantee condition {\rm (ii)}. 
Finally, conditions {\rm (iv)} and {\rm (v)} hold provided that $\epsilon_i>0$ is chosen small enough. 

This concludes the proof of Theorem \ref{th:ball}.

Corollary \ref{cor:ball} is a particular case of the following result.

\begin{corollary}\label{cor:ball-2}
Every open connected orientable smooth surface $S$ of finite topology admits a complex structure $J$ 
such that the open Riemann surface $R=(S,J)$ admits a complete holomorphic embedding $f\colon R\hookrightarrow\c^2$ 
whose image $f(R)$ lies in the ball $\b^2$ and contains any given countable subset of $\b^2$.  In particular, 
$f(R)$ can be made dense in $\b^2$.
\end{corollary}
\begin{proof}
Let $S$ be an open connected orientable smooth surface of finite topology, and let $A\subset\b^2$ be a countable subset. 
Let $J_0$ be a complex structure on $S$ such that the open Riemann surface $R_0=(S,J_0)$ admits a proper holomorphic 
embedding $\phi\colon R_0\hra \c^2$; such $J_0$ exists by \cite{CerneForstneric2002MRL}
(see also \cite{AlarconLopez2013JGA} for the case of surfaces of infinite topology). 
Up to composing with an homothety we may assume that all the topology of $X:=\phi(R_0)$ is contained in $\b^2$, 
meaning that $X\cap\b^2$ is homemorphic to $X$ and $X\setminus \b^2$ consists of finitely many pairwise disjoint 
closed annuli, each one bounded by a Jordan curve in $b\b^2=\{z\in\c^2\colon |z|=1\}$. Theorem \ref{th:ball} applied to 
$X\subset\c^2$ and any compact subset $K$ of $X\cap\b^2$ which is a strong deformation retract of $X$ gives a 
Runge domain $\Omega\subset X$ containing $K$ and a complete holomorphic embedding $f\colon \Omega\to\b^2$ 
with $A\subset f(\Omega)$. Since $K\subset \Omega$, $K$ is homeomorphic to $X$, and $\Omega$ is Runge in $X$, 
we have that also $\Omega$ is homeomorphic to $X$, and hence to $R_0=(S,J_0)$. Thus, there is a complex structure 
$J$ on $S$ such that $R=(S,J)$ is diffeomorphic to $\Omega$. The open Riemann surface $R$ and the complete 
holomorphic embedding $f\colon R\to\c^2$ satisfy the conclusion of the corollary.
\end{proof}


\subsection*{Acknowledgements}
A.\ Alarc\'on is supported by the Ram\'on y Cajal program of the Spanish Ministry of Economy and Competitiveness
and by the MINECO/FEDER grant no. MTM2014-52368-P, Spain. 
F.\ Forstneri\v c is partially  supported  by the research program P1-0291 and the research grant 
J1-7256 from ARRS, Republic of Slovenia. 





\vspace*{0.5cm}

\noindent Antonio Alarc\'{o}n

\noindent Departamento de Geometr\'{\i}a y Topolog\'{\i}a e Instituto de Matem\'aticas (IEMath-GR), Universidad de Granada, Campus de Fuentenueva s/n, E--18071 Granada, Spain

\noindent  e-mail: {\tt alarcon@ugr.es}

\vspace*{0.5cm}
\noindent Franc Forstneri\v c

\noindent Faculty of Mathematics and Physics, University of Ljubljana, Jadranska 19, SI--1000 Ljubljana, Slovenia.

\noindent Institute of Mathematics, Physics and Mechanics, Jadranska 19, SI--1000 Ljubljana, Slovenia.

\noindent e-mail: {\tt franc.forstneric@fmf.uni-lj.si}

\end{document}